\documentclass[12pt]{article}
\usepackage{amsmath}
\usepackage[cp1250]{inputenc}
\usepackage[T1]{fontenc}
\usepackage{amssymb, latexsym}
\usepackage{amsthm}
\usepackage{amscd}
\usepackage{makeidx}
\usepackage{amsfonts,mathrsfs}
\usepackage{topcapt}
\usepackage{hyperref}
\hypersetup{colorlinks=true, debug=true, linkcolor=blue,
hypertexnames=false}
\let\originalforall=\forall
\renewcommand{\forall}{\mathop{\vcenter{\hbox{\Large$\originalforall$}}}}

\let\originalexists=\exists
\renewcommand{\exists}{\mathop{\vcenter{\hbox{\Large$\originalexists$}}}}
\newtheorem{de}{Definition}

\newtheorem{tw}[de]{Theorem}
\newtheorem{prop}[de]{Proposition}
\newtheorem{lem}[de]{Lemma}
\newtheorem{cor}[de]{Corollary}

\title{On the Gelfand space of the measure algebra on the circle group}
\begin{document}
\author{Przemysław Ohrysko \\
Institute of Mathematics, Polish Academy of Sciences\\
00-956 Warszawa, Poland\\
E-mail: p.ohrysko@gmail.com \and
Michał Wojciechowski\thanks{The research of this author has been supported by NCN grant no. N N201 607840}\\
Institute of Mathematics, Polish Academy of Sciences\\
00-956 Warszawa, Poland\\
E-mail: miwoj-impan@o2.pl}

\maketitle
\begin{abstract}
This paper is devoted to studying certain topological properties
of the maximal ideal space of the measure algebra on the circle
group. In particular, we focus on Cech cohomologies of this space.
Moreover, we show that the Gelfand space of $M(\mathbb{T})$ is not
separable. On the other hand, we give a direct procedure to
recover many copies of $\beta\mathbb{Z}$ in
$\mathfrak{M}(M(\mathbb{T}))$, but we also show that this is
result is not accessible in the most natural way (namely, by the
canonical mapping induced by homomorphism assigning to measure its
Fourier - Stieltjes transform).
\end{abstract}
\section{Introduction}
Let $M(\mathbb{T})$ denote the convolution algebra of complex,
Borel and regular measures on the circle group. Also, let
$M_{d}(\mathbb{T})$ be the closed subalgebra of $M(\mathbb{T})$
consisting of discrete (purely atomic) measures while
$M_{c}(\mathbb{T})$ denotes the closed ideal of continuous
(non-atomic) measures. It follows from the general Gelfand theory
that $\mathfrak{M}(M(\mathbb{T}))$ - the space of all maximal
ideals or equivalently the space of all multiplicative - linear
functionals is compact a Hausdorff space in the weak$^{\ast}$
topology (those facts are proved in any textbook on Banach
algebras - see for example \cite{r2} or \cite{ż}). We are going to
give proofs of the basic topological properties of this space
using analytical and algebraic methods. Let us recall some
standard notations and basic facts which will be used in the
sequel (for details see \cite{kat}, \cite{h} and \cite{r1}).
\\
For $\mu\in M(\mathbb{T})$ and $n\in\mathbb{N}$ we define the
$n$-th Fourier - Stieltjes coefficient of a measure $\mu$ by the
formula
\begin{equation*}
\widehat{\mu}(n)=\int_{0}^{2\pi}e^{-int}d\mu(t).
\end{equation*}
It is well - known that the assignment
$\mu\mapsto\widehat{\mu}(n)$ is a multiplicative - linear
functional for every $n\in\mathbb{Z}$. In the same manner, the
mapping $\mu\mapsto (\widehat{\mu})_{n=-\infty}^{\infty}$ is
continuous homomorphism from $M(\mathbb{T})$ to
$l^{\infty}(\mathbb{Z})$. We will need the classical Wiener's
lemma which connects Fourier - Stieltjes coefficients with atoms
of the associated measure.
\begin{tw}[Wiener's lemma]
Let $\mu\in M(\mathbb{T})$. Then
\begin{equation*}
\lim_{N\rightarrow\infty}\frac{1}{2N+1}\sum_{n=-N}^{N}|\widehat{\mu}(n)|^{2}=\sum_{\tau\in\mathbb{T}}|\mu(\{\tau\})|^{2}.
\end{equation*}
In particular, $\mu\in M_{c}(\mathbb{T})$ if and only if
\begin{equation*}
\lim_{N\rightarrow\infty}\frac{1}{2N+1}\sum_{n=-N}^{N}|\widehat{\mu}(n)|^{2}=0.
\end{equation*}
\end{tw}
The following proposition is an easy exercise from many textbooks
on Banach algebras (see \cite{kan} and $\cite{r2}$)
\begin{prop}\label{na}
Let $A,B$ be complex, unital, commutative Banach algebras and
$f:A\mapsto B$ a continuous homomorphism. Then the mapping
$f^{\ast}:\mathfrak{M}(B)\mapsto \mathfrak{M}(A)$ defined by the
formula
\begin{equation*}
f^{\ast}(\varphi)(a)=\varphi(f(a))\text{ for
$\varphi\in\mathfrak{M}(B)$}
\end{equation*}
is continuous. Moreover, if the homomorphism $f$ has dense range,
then $f^{\ast}$ is injective.
\end{prop}
We will also make use of notion of almost periodic sequence (for
details and proofs of the next facts check \cite{kan})
\begin{de}
We say that the sequence of complex numbers
$(a_{n})_{n=-\infty}^{\infty}$ is almost periodic, if
\begin{equation*}
\forall_{\varepsilon>0}\exists_{p(\varepsilon)\in\mathbb{N}}\forall_{I}\exists_{m\in
I}\forall_{n\in\mathbb{Z}}|a_{n+m}-a_{n}|<\varepsilon,\text{ where
$I$ is any interval in $\mathbb{Z}$ of length $p(\varepsilon)$}.
\end{equation*}
The set of all almost periodic sequences will be denoted by
$AP(\mathbb{Z})$.
\end{de}
The most important properties of almost periodic sequences are
summarized in the following theorem.
\begin{tw}\label{dom}
$AP(\mathbb{Z})$ is closed $^{\ast}$-subalgebra of
$l^{\infty}(\mathbb{Z})$.
\end{tw}
A very useful link between measures and almost periodic sequences
is given in the next proposition.
\begin{prop}\label{dysp}
The sequence of Fourier - Stieltjes coefficients of any discrete
measure on $\mathbb{T}$ is almost periodic.
\end{prop}
\section{Cech cohomologies}
We begin this section with a simple proposition (in fact we work
with $0$ - dimensional Cech cohomologies - see \cite{t}).
\begin{prop}
There are only countably many open and closed subsets in
$\mathfrak{M}(M(\mathbb{T}))$.
\end{prop}
\begin{proof}
By the Shilov idempotent theorem for every open and closed subset
$U$ in $\mathfrak{M}(M(\mathbb{T}))$ there exists an idempotent
measure $\mu\in M(\mathbb{T})$ such that $\widehat{\mu}=\chi_{U}$.
On the other hand, by Helson's theorem (see \cite{h}) which
characterises idempotents in $M(\mathbb{T})$, there are only
countably many idempotents in $M(\mathbb{T})$ (it follows from the
fact that compact subgroups of $\mathbb{T}$ are finite and there
are only countably many of them).
\end{proof}
Despite the last proposition, we will show at the end of Section
$3$ that $\mathfrak{M}(\mathbb{T})$ contains many copies of
$\beta\mathbb{Z}$ (which is extremally disconnected!) as closed
topological subspaces.
\begin{prop}\label{rozl}
There are a continuum of pairwise disjoint copies of
$\beta\mathbb{Z}$ in $\mathfrak{M}(M(\mathbb{T}))$.
\end{prop}
From general topology, if a compact space has a countable basis,
then it has power at most continuum. But $\beta\mathbb{Z}$ is a
topological space of power $2^{c}$ which based on two last
propositions gives the following corollary.
\begin{cor}
$\mathfrak{M}(M(\mathbb{T}))$ is not a totally disconnected
topological space.
\end{cor}
This corollary can also be proved in different way - it follows
from Shilov idempotent theorem that if the Gelfand space of
commutative Banach algebra with unit is totally disconnected, then
this Banach algebra is regular which is not the case for
$M(\mathbb{T})$.
\\
We move on to $1$ - dimensional Cech cohomologies (a more detailed
discussion of related topics is given in \cite{t}).
\begin{de}
Let $K$ be a compact Hausdorff space and let
\begin{equation*}
\mathrm{Exp}(C(K))=\{\exp(f):f\in C(K)\}.
\end{equation*}
Then $\mathrm{Exp}(C(K))$ is a closed subgroup of $G(C(K))$ (group
of invertible elements in $C(K)$). We define $1$ - dimensional
Cech cohomologies $H^{1}(K)$ as the quotient group
\begin{equation*}
H^{1}(K):=G(C(K))/\mathrm{Exp}(C(K)).
\end{equation*}
For a commutative, complex Banach algebra with unit $A$ we define
$1$ - dimensional cohomologies of $H^{1}(A)$ analogously:
\begin{equation*}
H^{1}(A):=G(A)/\mathrm{Exp}(A).
\end{equation*}
\end{de}
It is an elementary fact that the Gelfand transform
$a\mapsto\widehat{a}$ induces a homomorphism of $H^{1}(A)$ into
$H^{1}(\mathfrak{M}(A))$. However, the following theorem (not
elementary at all!), with its proofs given in \cite{a}, \cite{r}
or \cite{g} based on complex analysis in several variables, gives
a much stronger statement.
\begin{tw}[Arens-Royden]\label{arr}
If $A$ is a complex, commutative Banach algebra with unit, then
the map $H^{1}(A)\mapsto H^{1}(\mathfrak{M}(A))$ induced by the
Gelfand transform is an isomorphism.
\end{tw}
We will give two applications of the Arens-Royden theorem. The
first one concerns $\beta\mathbb{Z}$ (this fact is probably known
in algebraic topology but using the above methods we will obtain
it very quickly).
\begin{prop}
$H^{1}(\beta\mathbb{Z})$ is trivial.
\end{prop}
\begin{proof}
Let us take $(a_{n})_{n=-\infty}^{\infty}\in
G(l^{\infty}(\mathbb{Z}))$. Remembering that
\begin{equation*}
\sigma((a_{n})_{n=-\infty}^{\infty})=\overline{\{a_{n}:n\in\mathbb{Z}\}}
\end{equation*}
we conclude $a_{n}\neq 0$ for every $n\in\mathbb{Z}$. Hence, there
exists a sequence $(b_{n})_{n=-\infty}^{\infty}$ of complex
numbers such that $\exp(b_{n})=a_{n}$ for every $n\in\mathbb{Z}$.
We can also assume that the sequence of real numbers
$(\mathrm{Im}b_{n})_{n=-\infty}^{\infty}$ is bounded. Now, we have
$|a_{n}|=|\exp(b_{n})|=\exp(\mathrm{Re}b_{n})$ for
$n\in\mathbb{Z}$. The sequence $(a_{n})_{n=-\infty}^{\infty}$ is
bounded and separated from $0$ so we obtain
$c_{1}\leq\mathrm{Re}b_{n}\leq c_{2}$ for some constants
$c_{1},c_{2}\in\mathbb{R}$ and every $n\in\mathbb{Z}$ which proves
$(b_{n})_{n=-\infty}^{\infty}\in l^{\infty}(\mathbb{Z})$. Finally,
\begin{gather*}
\exp((b_{n})_{n=-\infty}^{\infty})=\sum_{k=0}^{\infty}\frac{(b_{n}^{k})_{n=-\infty}^{\infty}}{k!}=\\
=\left(\sum_{k=0}^{\infty}\frac{b_{n}^{k}}{k!}\right)_{n=-\infty}^{\infty}=\left(\exp(b_{n})\right)_{n=-\infty}^{\infty}=(a_{n})_{n=-\infty}^{\infty}
\end{gather*}
which shows that every invertible bounded sequence belongs to
$\mathrm{Exp}(l^{\infty}(\mathbb{Z}))$. This finishes the proof
with the aid of Arens-Royden theorem.
\end{proof}
The second application of Theorem $\ref{arr}$ relates to
$\mathfrak{M}(M(\mathbb{T}))$. It shows that this space is
extremely complicated from the point of view of algebraic
topology.
\begin{tw}
$H^{1}(M(\mathbb{T}))$ is uncountable.
\end{tw}
\begin{proof}
Let us take $\alpha\in\mathbb{T}$, $\alpha\notin\pi\mathbb{Q}$.
Then $\delta_{\alpha}\in G(M(\mathbb{T}))$. We will show that the
assumption $\delta_{\alpha}\in\mathrm{Exp}(M(\mathbb{T}))$ leads
to a contradiction. If $\mu\in M(\mathbb{T})$ is such that
$\delta_{\alpha}=\exp(\mu)$, then for every $n\in\mathbb{Z}$ we
have
\begin{equation*}
\exp(-in\alpha)=\widehat{\delta_{\alpha}}(n)=\exp(\widehat{\mu}(n)).
\end{equation*}
Then,
\begin{equation*}
-in\alpha=\widehat{\mu}(n)+2\pi il_{n}\text{ for some
$l_{n}\in\mathbb{Z}$}.
\end{equation*}
Moreover, we easily see that we can assume that $\mu$ is a
discrete measure. Indeed, if $\mu=\mu_{c}+\mu_{d}$ is the standard
decomposition of a measure $\mu$, then
\begin{equation*}
\delta_{\alpha}=\exp(\mu_{c})\ast\exp(\mu_{d})=(\exp(\mu_{c})-\delta_{0})\ast\exp(\mu_{d})+\exp(\mu_{d}).
\end{equation*}
However $\exp(\mu_{c})-\delta_{0}$ is a continuous measure and the
set of those measures is an ideal, which gives
$\exp(\mu_{c})-\delta_{0}=0$ and hence $\exp(\mu)=\exp(\mu_{d})$.
From Proposition $\ref{dysp}$ we know that a sequence
$(\widehat{\mu}(n))_{n=-\infty}^{\infty}$ is almost periodic so
there exists $m\in\mathbb{N_{+}}$ such that for every
$n\in\mathbb{Z}$ the following inequality holds
\begin{equation*}
|\widehat{\mu}(n+m)-\widehat{\mu}(n)|=|m\alpha+2\pi(l_{n+m}-l_{n})|<1.
\end{equation*}
This gives: $l_{n+m}-l_{n}=s$ for some $s\in\mathbb{N}$ and every
$n\in\mathbb{Z}$. Now, pick any $n_{0}\in\mathbb{Z}$ such that
$l_{n_{0}}\neq 0$. Then for $k\in\mathbb{N}$ we have
\begin{equation*}
l_{n_{0}+mk}=l_{n_{0}+(m-1)k}+s=...=l_{n_{0}}+ks.
\end{equation*}
This leads to
\begin{equation*}
|\widehat{\mu}(n_{0}+mk)|=|k(m\alpha-2\pi s)+n_{0}\alpha-2\pi
l_{n_{0}})|
\end{equation*}
and recalling $\alpha\notin\pi\mathbb{Q}$ we obtain that a
sequence $(\widehat{\mu}(n_{0}+mk))_{k=0}^{\infty}$ is unbounded
which is the announced contradiction.
\\
Finally, if $\delta_{\alpha}$ and $\delta_{\beta}$ belongs to the
same coset of $G(M(\mathbb{T}))$ with respect to
$\mathrm{Exp}(M(\mathbb{T}))$, then
$\delta_{\alpha-\beta}\in\mathrm{Exp}(M(\mathbb{T}))$ and hence
$\alpha-\beta\in\pi\mathbb{Q}$. This proves that
$H^{1}(M(\mathbb{T}))$ is uncountable since we cannot split
$\mathbb{T}$ into countably many countable parts.
\end{proof}
\section{Main results}
This section is devoted to proving two striking results concerning
$\mathfrak{M}(M(\mathbb{T}))$.
\subsection{Fourier - Stieltjes sequences and ultrafilters}
First one says that Fourier - Stieltjes sequences of measures from
$\mathbb{T}$ `glue` some ultrafilters. The proof is based on the
following proposition.
\begin{prop}\label{nr}
Let $f:M(\mathbb{T})\mapsto l^{\infty}(\mathbb{Z})$ be the
homomorphism given by the formula
$f(\mu)=(\widehat{\mu})_{n=-\infty}^{\infty}$. Then the dual
mapping
$f^{\ast}:\mathfrak{M}(l^{\infty}(\mathbb{Z}))=\beta\mathbb{Z}\mapsto
\mathfrak{M}(M(\mathbb{T}))$ is injective if and only if $f$ has
dense range.
\end{prop}
\begin{proof}
One implication (if $f$ has dense range... ) is covered by
Proposition $\ref{na}$ from the introduction.
\\
Assume now that $f^{\ast}$ is injective. Since we have an
involution on $M(\mathbb{T})$ given by $\mu\mapsto\widetilde{\mu}$
where $\widetilde{\mu}(E)=\overline{\mu(-E)}$ which has the
property
\begin{equation*}
\widehat{\widetilde{\mu}}(n)=\overline{\widehat{\mu}(n)}\text{ for
every $n\in\mathbb{Z}$}
\end{equation*}
we easily see that $f(M(\mathbb{T}))$ is a $^{\ast}$ subalgebra of
$l^{\infty}(\mathbb{Z})$. Recalling that $l^{\infty}(\mathbb{Z})$
is $^{\ast}$ - isometrically isomorphic to $C(\beta\mathbb{Z})$
(see for example \cite{kan}) we may treat $f(M(\mathbb{T}))$ as a
$^{\ast}$ - subalgebra of $C(\beta\mathbb{Z})$. By the Stone -
Weierstrass theorem it is enough to verify that $f(M(\mathbb{T}))$
separates points of $\beta\mathbb{Z}$, but this is exactly our
assumption. Indeed, let us take two distinct ultrafilters
$\varphi_{1},\varphi_{2}\in\beta\mathbb{Z}$. Then, by the
assumption $f^{\ast}(\varphi_{1})(\mu_{1})\neq
f^{\ast}(\varphi_{2})(\mu_{2})$ for some $\mu_{1},\mu_{2}\in
M(\mathbb{T})$. Equivalently, we have
\begin{equation*}
\varphi_{1}((\mu_{1})_{n=-\infty}^{\infty})\neq\varphi_{2}(((\mu_{2})_{n=-\infty}^{\infty})
\end{equation*}
which is the desired assertion by the definition of the Gelfand
transform of an element in Banach algebra.
\end{proof}
Now, we are in position to prove the main theorem.
\begin{tw}
Let $f:M(\mathbb{T})\mapsto l^{\infty}(\mathbb{Z})$ be a
homomorphism as in the previous proposition. Then $f^{\ast}$ is
not injective.
\end{tw}
\begin{proof}
By the Proposition $\ref{nr}$ it is enough to prove that
$f(M(\mathbb{T}))$ is not dense in $l^{\infty}(\mathbb{Z})$. This
fact was established in \cite{dr} using the notion of weakly
almost periodic sequence, but our proof is much more elementary.
Let us take the sequence $(a_{n})_{n=-\infty}^{\infty}$ defined as
follows: $a_{n}=1$ for $n\geq 0$ and $a_{n}=0$ for $n<0$. It is
easily seen that the sequence $(a_{n})_{n=-\infty}^{\infty}$ is
not almost periodic. Hence, by Theorem $\ref{dom}$ there exists
$c>0$ such that
\begin{equation}\label{dal}
\inf_{\mu\in
M(\mathbb{T})}\sup_{n\in\mathbb{Z}}|\widehat{\mu_{d}}(n)-a_{n}|\geq\inf_{(b_{n})\in
AP(\mathbb{Z})}\sup_{n\in\mathbb{Z}}|a_{n}-b_{n}|>c
\end{equation}
Let us assume the contrary, that $f(M(\mathbb{T}))$ is dense in
$l^{\infty}(\mathbb{Z})$ and fix $\delta>0$. Then there exists
$\mu\in M(\mathbb{T})$ such that
\begin{equation*}
\sup_{n\in\mathbb{Z}}|\widehat{\mu}(n)-a_{n}|<\delta.
\end{equation*}
Splitting $\mu$ into into its discrete $\mu_{d}$ and continuous
part $\mu_{c}$ we have for every $n\in\mathbb{Z}$
\begin{equation*}
\delta>|\widehat{\mu}(n)-a_{n}|=|\widehat{\mu_{c}}(n)+\widehat{\mu_{d}}(n)-a_{n}|\geq
||\widehat{\mu_{c}}(n)|-|\widehat{\mu_{d}}(n)-a_{n}||,
\end{equation*}
which gives for every $n\in\mathbb{Z}$
\begin{equation}\label{cc}
|\widehat{\mu_{c}}(n)|\geq |\widehat{\mu_{d}}(n)-a_{n}|-\delta.
\end{equation}
Now, by the Inequality ($\ref{dal}$), without losing generality,
there exists $n_{0}\in\mathbb{N}_{+}$ (if $n_{0}\in\mathbb{Z}_{-}$
the same argument works) such that
\begin{equation*}
|\widehat{\mu_{d}}(n_{0})-a_{n_{0}}|>\frac{c}{2}.
\end{equation*}
Now, we use the fact that the sequence
$(\widehat{\mu_{d}}(n))_{n=-\infty}^{\infty}$ is almost periodic
(Proposition $\ref{dysp}$). Hence, for fixed $\varepsilon>0$ we
can find in every interval in $\mathbb{Z}$ of the form
$(kp(\varepsilon),(k+1)p(\varepsilon)]$ an integer $m_{k}$ such
that
$|\widehat{\mu_{d}}(n_{0}+m_{k})-\widehat{\mu_{d}}(n_{0})|<\varepsilon$
for every $k\in\mathbb{N}$. Recalling that $a_{n}=1$ for $n\geq 0$
we obtain for every $k\in\mathbb{N}$
\begin{gather*}
|\widehat{\mu_{d}}(n_{0}+m_{k})-a_{n_{0}+m_{k}}|=|\widehat{\mu_{d}}(n_{0}+m_{k})-a_{n_{0}}|=\\
|\widehat{\mu_{d}}(n_{0}+m_{k})-\widehat{\mu_{d}}(n_{0})+\widehat{\mu_{d}}(n_{0})-a_{n_{0}}|\geq\\
|\widehat{\mu_{d}}(n_{0})-a_{n_{0}}|-|\widehat{\mu_{d}}(n_{0}+m_{k})-\widehat{\mu_{d}}(n_{0})|>\frac{c}{2}-\varepsilon.
\end{gather*}
This estimation together with ($\ref{cc}$) gives for every
$k\in\mathbb{N}$
\begin{equation*}
|\widehat{\mu_{c}}(n_{0}+m_{k})|>\frac{c}{2}-\varepsilon-\delta.
\end{equation*}
On the other hand, from Wiener's lemma (by passing to a
subsequence) we get
\begin{equation*}
\lim_{k\rightarrow\infty}\frac{1}{2kp(\varepsilon)+1}\sum_{n=-kp(\varepsilon)+n_{0}}^{kp(\varepsilon)+n_{0}}|\widehat{\mu_{c}}(n)|^{2}=0
\end{equation*}
But in every interval
$[-kp(\varepsilon)-n_{0},kp(\varepsilon)+n_{0}]$ there are $k$
integers for which the summed expression is greater then
$(\frac{c}{2}-\varepsilon-\delta)^{2}$. This gives
\begin{gather*}
\frac{1}{2kp(\varepsilon)+2n_{0}+1}\sum_{n=-kp(\varepsilon)-n_{0}}^{kp(\varepsilon)+n_{0}}|\widehat{\mu_{c}}(n)|^{2}\geq\\
\frac{1}{2kp(\varepsilon)+2n_{0}+1}k\left(\frac{c}{2}-\varepsilon-\delta\right)^{2}\geq\frac{(\frac{c}{2}-\varepsilon-\delta)^{2}}{3p(\varepsilon)+2n_{0}},
\end{gather*}
which is the desired contradiction.
\end{proof}
\subsection{Non-separability}
Now, we move on to the second main theorem of this paper, namely
that the maximal ideal space of the measure algebra on the circle
group is not separable.
\\
We shall need the following simple arithmetic lemma which follows
easily from the fact that every element in a lacunary sequence
with ratio at least $3$ is greater than twice of sum of all
previous elements in the sequence.
\begin{lem}\label{ari}
Let $(n_{k})_{k=1}^{\infty}$ be sequence of positive integers such
that $\frac{n_{k+1}}{n_{k}}\geq 3$ for every $k\in\mathbb{N}$ and
for any set $A\subset\{n_{k}:k\in\mathbb{N}\}$, let us write
$\widetilde{A}$ for the set defined as follows
\begin{equation*}
\widetilde{A}=\left\{\sum_{l=1}^{n}\varepsilon_{l}a_{l}:\varepsilon_{l}\in\{-1,0,1\},a_{l}\in
A, n\in\mathbb{N}\right\}.
\end{equation*}
With these notions, if $A\cap B$ is finite, then
$\widetilde{A}\cap\widetilde{B}$ is finite.
\end{lem}
We will also make use of a well-known observation due to
Sierpiński.
\begin{prop}\label{sier}
There exists uncountably many infinite subsets of positive
integers such that intersection of each two is finite.
\end{prop}
Now, we recall a few facts on Riesz products, that is continuous
probabilistic measure on the circle of group of the following form
\begin{equation*}
R(a_{k},n_{k})=\prod_{k=1}^{\infty}(1+a_{k}\cos (n_{k}t)),
\end{equation*}
where this infinite product is meant as weak$^{\ast}$ limit of
finite products. From the construction of Riesz products we have
(for simplicity we write $\mu=R(a_{k},n_{k})$ and
$A=\{n_{k}:k\in\mathbb{N}$\})
\begin{equation*}
S(\mu):=\{n\in\mathbb{Z}:\widehat{\mu}(n)\neq 0\}=\widetilde{A}.
\end{equation*}
For a sequence of natural numbers $(n_{k})_{k=1}^{\infty}$ we
assume $\frac{n_{k+1}}{n_{k}}\geq 3$ for every $k\in\mathbb{N}$
and from $(a_{k})_{k=1}^{\infty}$ we demand $-1<a_{k}\leq 1$ for
$k\in\mathbb{N}$. We will use the following strong result on Riesz
products (for a proof consult \cite{bm} or \cite{grmc}).
\begin{tw}[Brown,Moran]\label{brm}
If $(a_{k})_{k=1}^{\infty}$ is a sequence of real numbers
satisfying $-1<a_{k}\leq 1$ for $k\in\mathbb{N}$ with the property
\begin{equation*}
\forall_{n\in\mathbb{N}}\sum_{k=1}^{\infty}|a_{k}|^{n}=\infty
\end{equation*}
then the Riesz product $R(a_{k},n_{k})$ has all convolution powers
mutually singular.
\end{tw}
It is an elementary result from the general theory of Banach
algebras (see for example \cite{grmc}), that for Riesz products
satisfying the assumptions of Theorem $\ref{brm}$ we have
\begin{equation*}
\{z\in\mathbb{C}:|z|=1\}\subset\sigma(R(a_{k},n_{k})).
\end{equation*}
In fact, a much stronger result is true (see \cite{bbm}).
\begin{tw}[Brown,Bailey,Moran]\label{bbmo}
If $\mu\in M(\mathbb{T})$ is a hermitian measure with all
convolution powers mutually singular then
\begin{equation*}
\sigma(\mu)=\{z\in\mathbb{C}:|z|\leq r(\mu)\}.
\end{equation*}
\end{tw}
The theorem of Zafran will also be appropriate (check \cite{z}).
\begin{tw}[Zafran]
Let $\mathscr{C}=\{\mu\in
M_{0}(\mathbb{T}):\sigma(\mu)=\widehat{\mu}(\mathbb{Z})\cup\{0\}$.
\begin{enumerate}
    \item If
    $\varphi\in\mathfrak{M}(M_{0}(\mathbb{T}))\setminus\mathbb{Z}$,
    then $\varphi(\mu)=0$ for all $\mu\in\mathscr{C}$.
    \item $\mathscr{C}$ is closed ideal in $M_{0}(\mathbb{T})$.
    \item $\mathfrak{M}(\mathscr{C})=\mathbb{Z}$.
\end{enumerate}
\end{tw}
We are ready now to prove the announced theorem.
\begin{tw}
$\mathfrak{M}(M(\mathbb{T}))$ contains uncountably many open
disjoint open subsets. In particular,
$\mathfrak{M}(M(\mathbb{T}))$ is not separable.
\end{tw}
\begin{proof}
Let us take any sequence $(n_{k})_{k=1}^{\infty}$ of positive
numbers such that $\frac{n_{k+1}}{n_{k}}\geq 3$ for
$k\in\mathbb{N}$. Then by Proposition $\ref{sier}$ there exists
uncountably many subsets $\{A_{\alpha}\}_{\alpha\in\mathbb{R}}$ of
$A=\{n_{k}:k\in\mathbb{N}\}$ with the property $A_{\alpha_{1}}\cap
A_{\alpha_{2}}$ is finite for any
$\alpha_{1},\alpha_{2}\in\mathbb{R}$, $\alpha_{1}\neq\alpha_{2}$.
For $\alpha\in\mathbb{R}$ we define the sequence
$(a_{k}^{\alpha})_{k=1}^{\infty}$ by the formula:
$a_{k}^{\alpha}=1$ for $k\in A_{\alpha}$ and $a_{k}^{\alpha}=0$
otherwise. We assign to every $A_{\alpha}$ the Riesz product
\begin{equation*}
\mu_{\alpha}:=R(a_{k}^{\alpha},n_{k})=\prod_{k=1}^{\infty}(1+a_{k}^{\alpha}\cos(n_{k}t)).
\end{equation*}
Then $S(\mu_{\alpha})=\widetilde{A_{\alpha}}$. By the assumption
and Lemma $\ref{ari}$ we obtain $S(\mu_{\alpha_{1}})\cap
S(\mu_{\alpha_{2}})$ is finite for
$\alpha_{1},\alpha_{2}\in\mathbb{R}$, $\alpha_{1}\neq\alpha_{2}$.
Now,
\begin{equation*}
S(\mu_{\alpha_{1}}\ast\mu_{\alpha_{2}})=S(\mu_{\alpha_{1}})\cap
S(\mu_{\alpha_{2}})\text{ is finite }.
\end{equation*}
Hence $\mu_{\alpha_{1}}\ast\mu_{\alpha_{2}}$ is a trigonometric
polynomial and it is easy to prove that all trigonometric
polynomials belong to $\mathscr{C}$. Let us define for $\mu\in
M(\mathbb{T})$
\begin{equation*}
\widetilde{S}(\mu)=\{\varphi\in\mathfrak{M}(M(\mathbb{T})):\widehat{\mu}(\varphi)\neq
0\}.
\end{equation*}
We also recall that
$\mathfrak{M}(M(\mathbb{T}))=\mathfrak{M}(M(\mathbb{T}))\setminus
h(M_{0}(\mathbb{T}))\cup\mathfrak{M}(M_{0}(\mathbb{T}))$ where
\begin{equation*}
h(M_{0}(\mathbb{T}))=\{\varphi\in\mathfrak{M}(M(\mathbb{T})):\varphi(\mu)=0\text{
for all }\mu\in M_{0}(\mathbb{T})\}.
\end{equation*}
Let us fix $\alpha_{1},\alpha_{2}\in\mathbb{R}$,
$\alpha_{1}\neq\alpha_{2}$ and take
$\varphi\in\mathfrak{M}(M(\mathbb{T}))\setminus\mathbb{Z}$. If
$\varphi\in\widetilde{S}(\mu_{\alpha_{1}})$, then
$\varphi(\mu_{\alpha_{1}}\ast\mu_{\alpha_{2}})=0$ which follows
from the theorem of Zafran if
$\varphi\in\mathfrak{M}(M_{0}(\mathbb{T}))$ and just from the
definition, if $\varphi\in\mathfrak{M}(M(\mathbb{T}))\setminus
h(M_{0}(\mathbb{T}))$. From this we get
$\varphi\notin\widetilde{S}(\mu_{\alpha_{2}})$ which we can
summarize as
\begin{equation}\label{tralal}
\widetilde{S}(\mu_{\alpha_{1}})\cap\widetilde{S}(\mu_{\alpha_{1}})\cap\mathfrak{M}(M(\mathbb{T}))\setminus\mathbb{Z}=\emptyset\text{
for }\alpha_{1}\neq\alpha_{2}.
\end{equation}
Using Theorem $\ref{bbmo}$ we know that there exists
$z\in\mathbb{C}\setminus\mathbb{R}$ in $\sigma(\mu_{\alpha})$ and,
recalling that
$\widehat{\mu_{\alpha}}(\mathbb{Z})\subset\mathbb{R}$, we are able
to find an open neighborhood $U$ of $z$ which does not intersect
the real line. Since
$\widehat{\mu_{\alpha}}:\mathfrak{M}(M(\mathbb{T}))\mapsto\mathbb{C}$
is a continuous function we get
$\widehat{\mu_{\alpha}}^{-1}(U)=T_{\alpha}$ is an open set
contained in $\mathfrak{M}(M(\mathbb{T}))\setminus\mathbb{Z}$. On
the other hand, $T_{\alpha}\subset\widetilde{S}(\mu_{\alpha})$ and
hence $T_{\alpha_{1}}\cap T_{\alpha_{2}}=\emptyset$ for
$\alpha_{1}\neq\alpha_{2}$ by ($\ref{tralal}$) which finishes the
proof.
\end{proof}
In the same manner we prove the following corollary which explains why determining spectra of measures is so difficult task.
\begin{cor}
There exist no countable set of multiplicative linear functionals on $M(\mathbb{T})$ such that the spectrum of any measure from $M(\mathbb{T})$ is a closure of the values of its Gelfand transform restricted to this set.
\end{cor}
We are ready now to give a proof of Proposition \ref{rozl} (we
follow notation from the last proof).
\\
Let $U,V\in \beta A_\alpha$ be two different ultrafilters. Then
there exists $X\in U$ and $Y\in V$ such that $X\cap Y=\emptyset$.
Let $\mu_X$ and $\mu_Y$ be the Riesz products built on the sets
$X$ and $Y$ respectively. Obviously we have
$\widehat{\mu}_X(n)=1/2$ for $n\in X$ and, by Lemma 15,
$\widehat{\mu}_Y(n)=0$ for sufficiently big $n\in X$. Therefore
$\lim_U \widehat{\mu}_X(n)=1/2$ while
$\lim_V\widehat{\mu}_X(n)=0$. Hence the map $\Lambda: \beta
A_\alpha\to \mathfrak{M}(M(\mathbb{T}))$ given by the formula
$\Lambda_\alpha (U)(\mu)=\lim_U\widehat{\mu}(n)$ is injective.
Then $\Lambda_\alpha(\beta A_\alpha)$ is a homeomorphic copy of
$\beta\mathbb{N}$. Exactly in the same way we prove that
$\Lambda_\alpha(A_\alpha)\cap\Lambda_\beta(A_\beta)=\emptyset$ for
$\alpha\neq\beta$.

\end{document}